\documentclass[11pt]{article}

\usepackage{setspace}
\usepackage{amsmath,amsfonts}
\usepackage{theorem}

\usepackage{epsfig}

\usepackage{xcolor}

\usepackage{lscape}

\usepackage{multirow}

\usepackage{rotating}

\usepackage{tikz}

\usepackage{makecell}
\newcolumntype{x}[1]{>{\centering\arraybackslash}p{#1}}
\usepackage[hyphens]{url}
\usepackage[breaklinks]{hyperref}
\hypersetup{
    colorlinks=true,       
    linkcolor=red,          
    citecolor=blue,        
    filecolor=magenta,      
    urlcolor=cyan           
}

\usepackage[sort&compress]{natbib}
\bibpunct{(}{)}{,}{a}{,}{,}

\theoremstyle{plain} { \theorembodyfont{\rmfamily}

\newtheorem{remark}{Remark}

}

\newtheorem{theorem}{Theorem}

\newcommand{\cF}{{\cal F}}
\newcommand{\cG}{{\cal G}}
\newcommand{\cT}{{\cal T}}

\newcommand{\cM}{{\cal M}}

\newcommand{\bF}{\mathbb{F}}
\newcommand{\bG}{\mathbb{G}}
\newcommand{\bZ}{\mathbb{Z}}
\newcommand{\bN}{\mathbb{N}}
\newcommand{\bR}{\mathbb{R}}

\newcommand{\bP}{\mathbb{P}}
\newcommand{\bxi}{\boldsymbol{\xi}}
\newcommand{\bdeta}{\boldsymbol{\eta}}
\newcommand{\prob}{\mathbb{P}}
\newcommand{\expect}{\mathbb{E}}

\newcommand{\vecr}{\mathbf{r}}

\newcommand{\SEP}{\mathrm{SEP}}
\newcommand{\hit}{\mathcal{H}}

\begin{document}
\title{Two explicit Skorokhod embeddings for simple symmetric random walk\thanks{Xue Dong He thanks research funds from Columbia University and The Chinese University of Hong Kong. Jan Ob{\l}{\'o}j thanks CUHK for hosting him as a visitor in March 2013 and gratefully acknowledges support from the ERC (grant no.\ 335421) under the EU's $7^{\textrm{th}}$ FP, and from St John's College in Oxford. The research of Xun Yu Zhou was supported by grants from Columbia University, the Oxford--Man Institute of Quantitative Finance and Oxford--Nie Financial Big Data Lab. The authors are grateful to anonymous reviewers for their comments and in particular express thanks to one reviewer who has suggested the current method of proof of Theorem \ref{thm:reformulate} to replace our original one and made comments which led to our Remark \ref{rk:nonuniqueAY}.}
}

\author{Xue Dong He\thanks{Corresponding Author. Department of Systems Engineering and Engineering Management, The Chinese University of Hong Kong, Hong Kong. Email: \texttt{xdhe@se.cuhk.edu.hk}.}
\and Sang Hu\thanks{School of Science and Engineering, The Chinese University of Hong Kong, Shenzhen, China 518172.  Email: \texttt{husang@cuhk.edu.cn}.}
\and
Jan Ob{\l}{\'o}j\thanks{Mathematical Institute and St John's College, University of Oxford, Oxford, UK. Email: \texttt{Jan.Obloj@maths.ox.ac.uk}.}
\and Xun Yu Zhou\thanks{Department of Industrial Engineering and Operations Research, Columbia University, New York, New York 10027.  Email: \texttt{xz2574@columbia.edu}.}}

\maketitle

\begin{abstract}

Motivated by problems in behavioural finance, we provide two explicit constructions of a randomized stopping time which embeds a given centered distribution $\mu$ on integers into a simple symmetric random walk in a uniformly integrable manner.
Our first construction has a simple Markovian structure: at each step, we stop if an independent coin with a state-dependent bias returns tails. Our second construction is a discrete analogue of the celebrated Az\'ema--Yor solution and requires independent coin tosses only when excursions away from maximum breach predefined levels. Further, this construction maximizes the distribution of the stopped running maximum among all uniformly integrable embeddings of $\mu$.

\medskip


\end{abstract}

\section{Introduction}\label{se:Introduction}

We contribute to the literature on Skorokhod embedding problem (SEP). The SEP, in general, refers to the problem of finding a stopping time $\tau$, such that a given stochastic process $X$, when stopped, has the prescribed distribution $\mu$: $X_\tau\sim \mu$. When such a $\tau$ exists we say that $\tau$ embeds $\mu$ into $X$. This problem was first formulated and solved by \citet{Skorokhod:65} when $X$ is a  standard Brownian motion. It has remained an active field of study ever since, see \cite{Obloj2004:SkorokhodEmbedding} for a survey, and has recently seen a revived interest thanks to an intimate connection with the Martingale Optimal Transport, see \cite{BeiglbockCoxHuesmann} and the references therein.

In this paper, we consider the SEP for the simple symmetric random walk.
Our interest arose from a casino gambling model of \citet{Barberis2012:Casino} in which the gambler's cumulative gain and loss process is modeled by a random walk. The gambler has to decide when to stop gambling and her preferences are given by cumulative prospect theory \citep{TverskyKahneman1992:CPT}. Such preferences lead to dynamic inconsistency, so this optimal stopping problem cannot be solved by the classical Snell envelop and dynamic programming approaches. By applying the Skorokhod embedding result we obtain here, \citet{HeEtal2014:StoppingStrategies} convert the optimal stopping problem into an infinite-dimensional optimization problem, find the (pre-committed) optimal stopping time, and study the gambler's behavior in the casino.

To discuss our results, let us first introduce some notation. We let $X=S=(S_t:t\geq 0)$ be a  simple symmetric random walk defined on a filtered probability space $(\Omega,\cF,\bF,\prob)$, where $\bF=(\cF_t)_{t\geq 0}$. We work in discrete time so here, and elsewhere, $\{t\geq 0\}$ denotes $t\in \{0,1,2,\ldots\}$. We let $\cT(\bF)$ be the set of $\bF$--stopping times and say that $\tau\in \cF(\bF)$ is uniformly integrable (UI) if the stopped process $(S_{t\wedge \tau}:t\geq 0)$ is UI. Here, and more generally when considering martingale, one typically restricts the attention to UI stopping times to avoid trivialities and to obtain solutions which are of interest and use. We let $\bZ$ denote the set of integers and  $\cM_0(\bZ)$ the set of probability measures on $\bZ$ which admit finite first moment and are centered.
Our prime interest is in stopping times which solve the SEP:
\begin{equation}
\label{eq:SEPdef}
\SEP(\bF,\mu):= \left\{\tau\in \cT(\bF): S_\tau\sim \mu \textrm{ and }\tau \textrm{ is UI}\right\}.
\end{equation}
Clearly if $\SEP(\bF,\mu)\neq \emptyset$ then $\mu\in \cM_0(\bZ)$. For embeddings in a Brownian motion, the analogue of \eqref{eq:SEPdef} has a solution if and only if $\mu\in \cM_0(\bR)$. However, in the present setup, the reverse implication depends on the filtration $\bF$. If we consider the natural filtration $\bF^S=(\cF^S_t:t\geq 0)$ where $\cF^S_t=\sigma(S_u:u\leq t)$, then
\citet{CoxObloj2008:ClassesofMeasures} showed that the set of probability measures $\mu$ on $\mathbb{Z}$ for which $\SEP(\bF^S,\mu)\neq \emptyset$ is a fractal subset of $\cM_0(\bZ)$. In contrast, \citet{Rost1971:Markoff} and \citet{Dinges1974:Stopping}
showed how to solve the SEP using randomized stopping times, so that if $\bF$ is rich enough then $\SEP(\bF,\mu)\neq \emptyset$ for all $\mu\in \cM_0(\bZ)$.
We note also that the introduction of external randomness is natural from the point of view of applications. In the casino model of \citet{Barberis2012:Casino} mentioned above, \citet{HeHuOblojZhou2016:Randomization} and \citet{HendersonHobsonTse14} showed that the gambler is strictly better off when she uses extra randomness, such as a coin toss, in her stopping strategy instead of relying on $\tau\in \cT(\bF^S)$. Similarly, randomized stopping times are useful in solving optimal stopping problems, see e.g.\ \citet{BelomestnyKraetschmer2014:OptimalStopping} and \citet{HeEtal2014:StoppingStrategies}.

Our contribution is to give two new constructions of $\tau\in \SEP(\bF,\mu)$ with certain desirable optimality properties. Our first construction, in Section \ref{se:RanodmizedPI} below, has minimal (Markovian) dependence property: a decision to stop only depends on the current state of $S$ and an independent coin toss. The coins are suitably biased, with state dependent probabilities, which can be readily computed using an explicit algorithm we provide. Such a strategy is easy to compute and easy to implement, which is important for applications, e.g.\ to justify its use by economic agents, see \citet{HeHuOblojZhou2016:Randomization}. We also link our construction to the embedding in \citet{CoxHobsonObloj:11} and show that the former corresponds to a suitable projection of the latter.
Our second construction, presented in Section \ref{se:AYlikeStoppingTimes}, is a discrete--time analogue of the \citet{AzemaYor1979} embedding. It is also explicit and its first appeal lies in the fact that it only stops when the loss, relative to the last maximum, gets too large. It also has an inherent probabilistic interest: it maximizes $\bP(\max_{t\leq \tau}S_t\geq x)$, simultaneously for all $x\in \bR$, among all $\tau\in \cT(\bF)$, attaining the classical \citet{BlackwellDubins1963:AConverse} bound for $x\in \bN$. We conclude the paper with explicit examples worked out in Section \ref{se:Example}.

\section{Randomized Markovian Solution to the SEP}\label{se:RanodmizedPI}
To formalize our first construction, consider $\vecr=(r^x:x\in \bZ)\in [0,1]^{\bZ}$ and a family of Bernoulli random variables $\bxi=\{\xi_{t}^x\}_{t\ge 0,x\in\mathbb{Z}}$ with $\bP(\xi_{t}^x=0)=r^{x}=1-\bP(\xi_{t}^x=1)$, which are independent of each other and of $S$. Each $\xi_{t}^x$ stands for the outcome of a coin toss at time $t$ when $S_t=x$ with $1$ standing for heads and $0$ standing for tails. To such $\bxi$ we associate
\begin{align}\label{eq:StoppingTimeCoinToss}
\tau(\vecr):=\inf\{t\ge 0: \xi_{t}^{S_t}=0\},
\end{align}
which is a stopping time relative to $\bF^{S,\bxi}=(\cF^{S,\bxi}_t:t\geq 0)$ where $\cF^{S,\bxi}_t:=\sigma(S_s,\xi_{s}^{S_s},s\leq t)$. The decision to stop at time $t$ only depends on the state $S_t$ and an independent coin toss. Accordingly, we  refer to $\tau(\vecr)$ as a {\em randomized Markovian stopping time}. It is clear however that the distribution of $S_{\tau(\vecr)}$ is a function of $\vecr$ and does not depend on the particular choice of the random variable $\bxi$.
The following shows that such stopping times allow us to solve \eqref{eq:SEPdef} for all $\mu\in \cM_0(\bZ)$.
\begin{theorem}\label{thm:reformulate}
	For any $\mu \in \mathcal{M}_0(\mathbb{Z})$, there exists $\vecr_\mu\in [0,1]^\bZ$ such that $\tau(\vecr_\mu)$ solves $\SEP(\bF^{S,\bxi},\mu)$.
\end{theorem}
\subsection{Proof of Theorem \ref{thm:reformulate}}\label{sec:proofMarkovian}
We establish the theorem by embedding our setup into a Brownian setting and then using Theorem 5.1 in \citet{CoxHobsonObloj:11}.  We reserve $t$ for the discrete time parameter and use $u\in [0,\infty)$ for the continuous time parameter. We assume our probability space $(\Omega,\cF,\prob)$ supports a standard Brownian motion $B=(B_u)_{u\geq 0}$. We let $\bG=(\cG_u)_{u\geq 0}$ denote its natural filtration taken right-continuous and complete and $L_u^y$ denote its local time at time $u\geq 0$ and level $y\in \bR$. Recall that $\mu\in \mathcal{M}_0(\mathbb{Z})$ is fixed.
\citet{CoxHobsonObloj:11} show that there exists a measure $m$ on $\mathbb{R}$ such that, for a Poisson random measure $\Delta^m$ with intensity $du \times m(dx)$, independent of $B$,
$$T^m=\inf\{u\geq 0: \Delta^m(R_u)\geq 1 \},\quad \text{where }R_u=\{(s,y): L_u^y>s\},$$
is minimal and embeds $\mu$, i.e., $B_{T^m}\sim \mu$ and $(B_{T^m\land u}: u \geq 0)$ is uniformly integrable, the latter being equivalent to minimality of $T^m$, see \citet[Sec.~8]{Obloj2004:SkorokhodEmbedding}. Moreover, by the construction in \citet{CoxHobsonObloj:11}, $m(I^c)=+\infty$ for any interval $I$ that contains the support of $\mu$. For $\mu \in \mathcal{M}_0(\mathbb{Z})$ it follows that $m$ is a measure on $\bZ$: $m(dy)=\sum_{x\in \bZ}m^x\delta_{x}(dy)$ ; otherwise, $\Delta^m(R_u)$ can possibly hit $1$ when the local time $L_u^y$ accumulates at certain non-integer level $y$, in which case $B_{T^m}$ takes value $y$. In addition, 
if $\mu$ has support bounded from above, i.e., $\mu([\bar{x},\infty))=\mu(\{\bar{x}\})>0$ for a certain $\bar x\in\mathbb{Z}$, then $m^{\bar{x}}=\infty$ and $T^m\leq \inf\{u\geq 0: B_u \geq \bar{x}\}$, with analogous expressions when the support is bounded from below.
We note that $N^x_u= \Delta^m(\{(s,x): s\leq u\})$ is a Poisson process with parameter $m^x$, $x\in \bZ$ and its first arrival time $\rho^x=\inf\{u\ge 0: N^x_u\geq 1\}$ is exponentially distributed with parameter $m^x$. We can now rewrite the embedding time as
$$T^m=\inf\{u\geq 0: L_u^x\geq \rho^x \text{ for some }x\in \bZ\}.$$
Consider now consecutive hitting times of integers
$$ \sigma_0=0,\quad \sigma_t=\inf\left\{u\geq \sigma_{t-1}: B_u\in \bZ\setminus\{B_{\sigma_{t-1}}\}\right\},\quad t=1,2,\ldots,$$
and note that $X_t:= B_{\sigma_t}$ is a simple symmetric random walk. Recall that the measure $dL^x_u$ is supported on $\{u: B_u=x\}$. This implies a particularly simple structure of the stopping time $T^m$. First, note that $T^m\neq \sigma_t$ unless $m^{X_t}=\infty$. Then, let us describe $T^m$ conditionally on $T^m>\sigma_t$. In particular, $\rho^{X_t}>L^{X_t}_{\sigma_t}$ and $\rho^{X_t}-L^{X_t}_{\sigma_t}$ is again exponential with parameter $m^{X_t}$. $T^m$ happens in $(\sigma_t,\sigma_{t+1})$ if and only if the local time accumulated at level $X_t$ is greater than this exponential variable. Clearly this event depends on the past only through the value of $X_t$. More formally, considering the new Brownian motion $B^t_u=B_{u+\sigma_t}-B_{\sigma_t}$, $u\geq 0$, and denoting its local time in zero as $L^{(0,t)}_u$ we see that $T^m\in (\sigma_t,\sigma_{t+1})$ if and only if $\{L^{(0,t)}_{\sigma_{t+1}}\geq \rho^{X_t}-L^{X_t}_{\sigma_t}\}$ which, conditionally on $X_t=x$, is independent of $\cF_{\sigma_t}$ and has probability which only depends on $x$. Further, in this case $B_{T^m}=X_t$.
If we let $\tau_{B,X}=t$ on $\{\sigma_t\leq T^m<\sigma_{t+1}\}$ then it follows that $\tau_{B,X}$ is a stopping time relative to the natural filtration of $X$ enlarged with a suitable family of independent random variables, it has the precise structure of $\tau(\vecr)$ in \eqref{eq:StoppingTimeCoinToss}, embeds $\mu$ and is UI. More precisely, using the space homogeneity of Brownian motion, we can define the probabilities using just the local time in zero, and it follows that if we take
$$r^x=\bP(L^0_{\sigma_1}>\rho^x)$$
then $\tau(\vecr)\in \SEP(\bF^{S,\bxi},\mu)$ as required.

\begin{remark}\label{rk:maximalr}
 We note that by following the methodology in \citet{CoxHobsonObloj:11} one could write a direct proof of Theorem \ref{thm:reformulate}, albeit longer and more involved than the one above. In particular, it is insightful to point out that if $\mu$ has a finite support -- $\mu([\underline{x},\bar{x}])=1$ with $\mu(\{\underline{x}\}) > 0$, $\mu(\{\bar{x}\}) > 0$ for some $\underline{x}<0<\bar{x}$ -- then $\vecr$ as constructed above
 can be shown to be the maximal element in the set
 \begin{equation*}
\mathcal{R}_{\mu} = \{\vecr \in [0,1]^{\mathbb{Z}}: r^i=1 \text{ if }i\notin (\underline{x},\bar{x}) \text{ and } \mathbb{P}(S_{\tau(\vecr)} = i) \leq \mu(\{i\}) \text{ if } i \in (\underline{x},\bar{x}) \}.
\end{equation*}
\end{remark}

\subsection{Algorithmic computation of the stopping probabilities $\vecr_\mu$}
\label{subse:ConstructionRandPI}
In this section, we work under the assumptions of Theorem \ref{thm:reformulate} and provide an algorithmic method for computing $\{r^i\}_{i\in \mathbb{Z}}$ obtained therein.
We let $\tau=\tau(\vecr_\mu)$ and $g^i$ denote the expected number of visits of $S$ to state $i$ strictly before $\tau$, i.e.:
\begin{align}\label{eq:ProbReaching}
g^i:=\expect\left[\sum_{t=0}^{\tau-1}\mathbf{1}_{S_t=i}\right]=
\sum_{t=0}^\infty\mathbb{P}\left(\tau>t,S_{t} = i\right).
\end{align}
Denote $a^+:= \max\{a,0\}$. It is a matter of straightforward verification to check that for any $i\leq 0\leq j$ the processes
$$(i-S_t)^+-\frac12 \sum_{u=0}^{t-1}\mathbf{1}_{S_u=i},\quad (S_t-j)^+-\frac12 \sum_{u=0}^{t-1}\mathbf{1}_{S_u=j},\quad t\geq 0,$$
are martingales. To compute $g^i$, we then apply the optional sampling theorem at $\tau\land t$ and let $t\to \infty$. Using the fact that $\{S_{\tau\land t}\}$ is a UI family of random variables together with monotone convergence theorem, we deduce that
\begin{align}
g^i &= 2 \expect[(S_\tau-i)^+] = 2\sum_{k=i}^{+\infty} \prob(S_\tau \ge k+1) ,\quad i=0,1,2,\dots,\label{eq:GiPositive}\\
g^i &= 2 \expect[(i-S_\tau)^+] = 2\sum_{k=-\infty}^i \prob(S_\tau \le k-1),\quad i=0,-1,-2,\dots\label{eq:GiNegative}
\end{align}
Writing $p^i:= \mu(\{i\})=\prob(S_{\tau}=i)$, we now compute
\begin{align*}
\begin{split}
p^i &= \mathbb{P}\left(S_{\tau} = i\right)
=\sum_{t=0}^{\infty} \mathbb{P}\left(\tau = t, S_{t} = i\right)\\
& = \sum_{t=0}^{\infty} \mathbb{P}\left(\xi_{u,S_u} = 1, u=0,1,\dots, t-1,\xi_{t,S_t}=0, S_{t} = i\right)\textrm{, and by conditioning}\\
& = \sum_{t=0}^{\infty} \mathbb{P}\left(\xi_{u,S_u} = 1, u=0,1,\dots, t-1,S_{t} = i\right)r^i
\\
&= r^i \sum_{t=0}^{\infty} \mathbb{P}\left(\tau\ge t,S_{t} = i\right)
= r^i(g^i+p^i).
\end{split}
\end{align*}
Therefore, if $p^i+g^i>0$, we must have $r^{i} = \frac{p^{i}}{p^{i} + g^{i}}$. If $p^i>0$ and $g^i=0$, which is the case if and only if $i$ is on the boundaries of the support, we have $r^i=1$, i.e., we have to stop instantly. If $p^i+g^i=0$, then $p^i=g^i=0$ and this can only happen for states outside the boundaries of the support. In this case, we set $r^i=1$, which is consistent with the characterisation in Remark \ref{rk:maximalr}.
Thus $\vecr_\mu=(r^i)$ in Theorem \ref{thm:reformulate} is given by
\begin{align}\label{eq:Computeri}
r^{i} = \frac{p^{i}}{p^{i} + g^{i}}\mathbf 1_{\{p^{i} + g^{i}>0\}} +\mathbf 1_{\{p^{i} + g^{i}=0\}},\quad i\in \mathbb{Z},
\end{align}
where $p^i=\mu(\{i\})$ and $g^i$ can be calculated from \eqref{eq:GiPositive} and \eqref{eq:GiNegative}.
This can be seen as the equation on the bottom of page S22 in \cite{CoxHobsonObloj:11} specialised to our setup. While that equation is only argued heuristically therein, in our setup we can give it a rigorous meaning and proof.

\section{Randomized Az\'ema-Yor solution to the SEP}\label{se:AYlikeStoppingTimes}
Let us recall the celebrated Az\'ema-Yor solution to the SEP for a standard Brownian motion $(B_u:u\geq 0)$. As above, we reserve $t$ for the discrete time parameter and use $u\in [0,\infty)$ for the continuous time parameter. To a centered probability measure $\mu$ on $\mathbb{R}$ we associate its barycenter function
\begin{align}\label{eq:barycentrefunction}
\psi_\mu(x):=\frac{1}{\bar \mu(x)}\int_{[x,+\infty)}y\mu(dy),\quad x\in \mathbb{R},
\end{align}
where $\bar \mu(x):=\mu\big([x,+\infty)\big)$ and $\psi_\mu(x):=x$ for $x$ such that $\mu([x,+\infty))=0$. We let $b_\mu$ denote the right-continuous inverse of $\psi_\mu$; i.e., $b_\mu(y):=\sup\{x:\psi_{\mu}(x)\le y\}$, $y\ge 0$. Then
\begin{align}\label{eq:AYcontdef}
T^{AY}_\mu:=\inf\{u\ge 0: B_u\leq b_\mu(B^*_u)\},\quad \textrm{ where } B^*_u:= \sup_{s\leq u} B_s,
\end{align}
satisfies $B_{T^{AY}_\mu}\sim \mu$ and $(B_{u\wedge T^{AY}_\mu}:u \geq 0)$ is UI. Furthermore, for any other such solution $\widetilde T$ to the SEP, and any $x\geq 0$, we have $\bP(B^*_{\widetilde T}\geq x)\leq \bP(B^*_{T^{AY}_\mu}\geq x)=\bar\mu^{HL}(x)$; hence $T^{AY}_\mu$ maximizes the distribution of the maximum in the stochastic order. Here $\bar \mu^{HL}$ is the Hardy-Littlewood transform of $\mu$ and the bound $\bar\mu^{HL}(x)$ is due to \cite{BlackwellDubins1963:AConverse}, and has been extensively studied since; see e.g. \cite{CarraroElKarouiObloj:09} for details.

A direct transcription of the Az\'ema-Yor embedding to the context of a simple symmetric random walk only works for measures $\mu\in\cM_0(\bZ)$ for which $\psi_\mu(x)\in \bN$ for all $x\in\bR$, which is a restrictive condition; see \citet{CoxObloj2008:ClassesofMeasures} for details.
For a general $\mu\in \cM_0(\bZ)$ we should seek instead to emulate the structure of the stopping time: the process stops when its drawdown hits a certain level, i.e., when $B_u^*-B_u\ge B_u^*-b_\mu(B_u^*)$, or when it reaches a new maximum at which time a new maximum drawdown level is set, whichever comes first. Only in general, we expect to use an independent randomization when deciding to stop or to continue an excursion away from the maximum. Surprisingly, this can be done explicitly and the resulting stopping time maximizes stochastically the distribution of the running maximum of the stopped random walk among all solutions in \eqref{eq:SEPdef}.

Before we state the theorem we need to introduce some notation. Let $\mu \in \mathcal{M}_0(\mathbb{Z})$ and denote $\bar x:=\sup\{n\in\mathbb Z:\mu(\{n\})>0\}$ and $\underline{x}:=\inf\{n\in\mathbb{Z}:\mu(\{n\})>0\}$ the bounds of the support of $\mu$.
The barycentre function  $\psi_\mu$ is piece-wise constant on $(-\infty,\bar x]$ with jumps in atoms of $\mu$, is non-decreasing, left-continuous with $\psi_\mu(x)>x$ for $x<\bar x$, $\psi_\mu(x)=x$ for $x\in[\bar x,+\infty)$, and $\psi_\mu(-\infty)=0$. The inverse function $b_\mu$ is right-continuous, non-decreasing with $b_\mu(0)=\underline{x}$, and is integer valued on $(0,\bar x)$. Further, $b_\mu(y)<y$ for $y<\bar x$ and $b_\mu(y)=y$ for $y\in[\bar x,+\infty)$; in particular $b_\mu(n)\leq n-1$ for $n\in \bZ\cap [0,\overline{x})$. Moreover, for any $n\in \bZ\cap [\underline{x},\overline{x}]$, $\mu(\{n\})>0$ if and only if $n$ is in the range of $b_\mu$; consequently, $\mu(\{y\})=0$ for any $y$ that is not in the range of $b_\mu$.

For each $1\le n<\bar x$, $\{b_\mu(y):y\in[n,n+1]\} \cap (-\infty, n]$ is a nonempty set of finitely many integers which we rank and denote as
$x^n_1>x^n_2>\dots>x^n_{m_n+1}$. Similarly, we rank $\{b_\mu(y):y\in[0,1]\}$, which is a nonempty set of finitely many integers $x^0_1>x^0_2>\dots > x^0_{m_0+1}$ if $b_\mu(0)>-\infty$ and a set of countably many integers $x^0_1>x^0_2>\dots$ otherwise. Then, for each $1\le n< \bar x$, $x^{n-1}_{1} = x^{n}_{m_{n}+1}=b_\mu(n)\le n-1$. Note that we may have $m_n=0$, in which case $x^{n-1}_1=x^n_{m_n+1}=x^n_1$.
For each $1\le n<\bar x$ and for $n=0$ when $\underline{x}>-\infty$, define
\begin{align}
\begin{split}\label{eq:DefineRho}
\Gamma^n:&=\bar \mu\left(x^{n}_{m_n+1}\right)\frac{\psi_{\mu}(x^{n}_{m_n+1})-x^{n}_{m_n+1}}{n-x^{n}_{m_n+1}},\\
g^n_k:&=\frac{n+1-x^{n}_{k}}{\Gamma^n}\mu(\{x^n_k\}),\quad k=2,3,\dots, m_n,\\
g^n_{m_n+1}:&=\frac{n+1-x^{n}_{m_n+1}}{\Gamma^n}\left[\mu(\{x^{n}_{m_n+1}\}) - \bar \mu\left(x^{n}_{m_n+1}\right)\frac{n-\psi_\mu(x^{n}_{m_n+1})}{n-x^{n}_{m_n+1}}\right],\\
f^n_{m_n+1}:&=0, \quad f^n_{k} = f^n_{k+1}+g^n_{k+1},\quad k=1,2,\dots, m_n,\quad f^n_0:=1.
\end{split}
\end{align}
When $\underline{x}=-\infty$, define
\begin{align}\label{eq:DefineRho0}
g^0_k = (1-x^{0}_{k})\mu(\{x^0_k\}),\; k\ge 2,\quad f^0_k = \sum_{i=k+1}^\infty g^0_i,\; k\ge  1,\quad f^0_0:=1.
\end{align}
Then, as we will see in the proof of Theorem \ref{th:SkorokhodembeddingAYstoppingtime}, for each $1\le n<\bar x$ and for $n=0$ when $\underline{x}>-\infty$, $\rho^n_k:=1-(f^n_k/f^n_{k-1})$ is in $[0,1)$ for $k=1,\dots, m_n$ and we let $\rho^n_{m_n+1}:=1$; when $\underline{x}=-\infty$, $\rho^0_k:=1-(f^0_k/f^0_{k-1})$ is in $(0,1)$ for each $k\ge 1$ and we set $m_0+1:=+\infty$. Let $\bdeta=(\eta^n_k: 0 \leq n<\bar x, k\in \mathbb{Z}\cap [1, m_n+1])$ be a family of mutually independent Bernoulli random variables, independent of $S$, with $\bP(\eta^n_k=0)=\rho^n_k=1-\bP(\eta^n_k=1)$. We let $S_t^*:=\sup_{r\leq t}S_r$ and define the enlarged filtration $\bF^{S,\bdeta}$ via $\cF_t^{S,\bdeta}:=\sigma(S_u,\eta^{S^*_u}_k:k\in \mathbb{Z}\cap [1, m_{S^*_u}+1], u\leq t)$.

We are now ready to define our Az\'ema--Yor embedding for $S$. It is an $\bF^{S,\bdeta}$--stopping time which, in analogy to \eqref{eq:AYcontdef}, stops when an excursion away from the maximum breaches a given level.
However, since the maximum only takes integer values, we emulate the behaviour of $B_{u\land T^{AY}_\mu}$ between hitting times of two consecutive integers in an averaged manner, using independent randomization.
Specifically, if we let $\hit_n:=\inf\{t\geq 0: S_t=n\}$ then after $\hit_n$ but before $\hit_{n+1}$ we may stop at each of $x^n_1>x^n_2>\ldots>x^n_{m_n}$ depending on the independent coin tosses $(\eta^n_k)$, while we stop a.s.\ if we hit $x^n_{m_n+1}$. If we first hit $n+1$ then a new set of stopping levels is set. Finally, we stop upon hitting $\bar x$.
\begin{theorem}\label{th:SkorokhodembeddingAYstoppingtime}
	Let $\mu \in \mathcal{M}_0(\mathbb{Z})$ and $\tau^{AY}_\mu$ be given by
	\begin{equation*}
	\tau^{AY}_\mu:= \inf\left\{t\geq 0: S_t\leq x^{S^*_t}_k \textrm{ and }\eta^{S^*_t}_k=0 \textrm{ for some } k\in \mathbb{Z}\cap [1, m_n+1]\right\}\land \hit_{\bar x}.
	\end{equation*}
	Then $\tau^{AY}_\mu\in \SEP(\bF^{S,\bdeta},\mu)$ and
	for any $\sigma\in \SEP(\bF,\mu)$
	\begin{equation}\label{eq:AYoptimality}
	\bP(S^*_\sigma \geq n) \leq \bP(S^*_{\tau^{AY}_\mu} \geq n) = \bar \mu(b_\mu(n)) \frac{\psi_\mu(b_\mu(n)) - b_\mu(n)}{n-b_\mu(n)}
	,\quad n\in \bN,
	\end{equation}
	with the convention $\frac{0}{0}=1$.
\end{theorem}
The optimality property in \eqref{eq:AYoptimality} is analogous to the optimality of $T^{AY}_\mu$ in a Brownian setup, as described above and the bound in \eqref{eq:AYoptimality} coincides with $\bar \mu^{HL}(n)$.
\begin{remark}
Note that by considering our solution for $(-S_t)_{t\geq 0}$ we obtain a reversed Az\'ema--Yor solution which stops when the maximum drawup since the time of the historical minimum hits certain levels. It follows from \eqref{eq:AYoptimality} that such embedding maximizes the distribution of the running minimum in the stochastic order.
\end{remark}
\begin{remark}\label{rk:nonuniqueAY}
 We do not claim that $\tau^{AY}_\mu$ is the only embedding which achieves the upper bound in \eqref{eq:AYoptimality}. Our construction inherits the main structural property of the Az\'ema-Yor embedding for $B$: when a new maximum is hit, a lower threshold is set (which may depend on an independent randomisation) and we either stop when this threshold is hit or else a new maximum is set. This, in effect, averages out the behaviour of $B_{u\land T^{AY}_\mu}$ between hitting times of two consecutive integers. Instead, we could consider averaging out only the behaviour of $B_{u\land T^{AY}_\mu}$ between the first hitting time of an integer $n$ and the minimum between the hitting time of $n+1$ and the return of the embedded walk to $n$. The resulting embedding would have the following structure: each time the random walk is at its maximum, $S_t=S^*_t=n$ a (randomized) threshold is set and the walk stops if it hits this threshold. If not, a new instance of the threshold level is drawn when the walk returns to $n$. This is iterated until the walk is stopped at the current threshold level or when $n+1$ is hit which changes the distribution of the threshold level. This embedding appears less natural for us, having the casino gambling motivation in mind, however it should share the optimality property of $\tau^{AY}_\mu$ in \eqref{eq:AYoptimality}.
\end{remark}

\subsection{Proof of Theorem \ref{th:SkorokhodembeddingAYstoppingtime}}
	It is straightforward to verify that all the conclusions of the theorem hold for the case in which $\mu(\{0\})=1$, so we assume in the following that $\mu(\{0\})<1$ and thus $\underline{x}<0<\bar x$.
	Throughout the proof we let $\tau=\tau^{AY}_\mu$ and recall that $\hit_j = \inf\{t\ge 0: S_t = j\}$.
	We first prove constructively that $\rho^n_i$'s are well defined and the constructed stopping time $\tau$ embeds $\mu$ into the random walk. Specifically, we argue by induction that for any $0\le j<\bar x$ the stopping time $\tau$ satisfies
	\begin{align}\label{eq:AYlikestoppingtime}
	\begin{cases}
	\mathbb{P}(S_{\tau} = y \text{ and } \tau < \hit_{j+1}) = \mu(\{y\}), &\text{ for } y < x^{j}_1\\
	\mathbb{P}(S_{\tau} = y \text{ and } \tau < \hit_{j+1}) = \bar {\mu}(x^j_1) \frac{j+1-\psi_{\mu}(x^{j}_1)}{j+1-x^{j}_1},  &\text{ for } y = x^j_1\\
	\mathbb{P}(S_{\tau} = y \text{ and } \tau < \hit_{j+1}) = 0,  &\text{ for } y > x^j_1.
	\end{cases}
	\end{align}
For clarity of the presentation, somewhat lengthy and technical, proof of the above equalities is relegated to Section \ref{sec:proof_keyrep} below.
	
	Next, we show that $S_\tau\sim\mu$. If $\bar x=+\infty$, by the construction of $\tau$, we have
	\begin{align*}
	\prob(\tau=+\infty)\le \lim_{n\rightarrow+\infty}\prob(\tau\ge \hit_{n}) = \lim_{n\rightarrow+\infty}{\bar {\mu}(x^{n}_{m_n+1})}\frac{\psi_{\mu}(x^{n}_{m_n+1})-x^{n}_{m_n+1}}{n-x^{n}_{m_n+1}}=0,
	\end{align*}
	since $\lim_{n\rightarrow+\infty}x^{n}_{m_n+1} = \lim_{n\rightarrow+\infty}b_\mu(n)=+\infty$, $\lim_{x\rightarrow +\infty}\bar \mu(x)=0$, and $\psi_\mu(x^{n}_{m_n+1})\in [x^{n}_{m_n+1},n)$. Consequently, $\tau<+\infty$ a.s. and for any $y\in\mathbb Z$, $\mathbb{P}(S_{\tau} = y) =\lim_{n\rightarrow+\infty}\mathbb{P}(S_{\tau} = y , \tau < \hit_{n}) = \mu(\{y\})$. If $\bar x<+\infty$, by definition, $\tau\le \hit_{\bar x}<\infty$ a.s. Moreover, because \eqref{eq:AYlikestoppingtime} is true for $j=\bar x-1$, we have $\prob(S_\tau = y) = \prob(S_\tau = y,\tau < \hit_{\bar x}) = \mu(\{y\})$
	for any $y<x^{\bar x-1}_1$. From the definition of $x^{\bar x-1}_1$, we conclude that $b_\mu$ does not take any values in $(x^{\bar x-1}_1,\bar x)$, so $\mu(\{n\})=0$ for any integer $n$ in this interval. Since $S_\tau\leq \bar x$, it remains to argue $\prob(S_\tau = x^{\bar x-1}_1) = \mu(\{x^{\bar x-1}_1\})$. This follows from \eqref{eq:AYlikestoppingtime} with $j=\bar x-1$:
	\begin{align*}
	\prob(&S_\tau = x^{\bar x-1}_1) = \prob(S_\tau = x^{\bar x-1}_1,\tau < \hit_{\bar x}) \\
	&= \bar \mu(x^{\bar x-1}_1)\frac{\bar x-\psi_\mu(x^{\bar x-1}_1)}{\bar x-x^{\bar x-1}_1}=\frac{\bar x  \bar \mu(x^{\bar x-1}_1) - \sum_{n\ge x^{\bar x-1}_1}n\mu(\{n\})}{\bar x-x^{\bar x-1}_1}\\
	&=  \frac{\bar x  \left( \mu(\{x^{\bar x-1}_1\})+\mu(\{\bar x\})\right) - \left(x^{\bar x-1}_1\mu(\{x^{\bar x-1}_1\})+\bar x \mu(\{\bar x\})\right)}{\bar x-x^{\bar x-1}_1}=\mu(\{x^{\bar x-1}_1\}),
	\end{align*}
	where the fourth equality follows because $\mu(\{n\})=0$ for any $n> x^{\bar x-1}_1$ and $n\neq \bar x$.
	
	To conclude that $\tau\in \SEP(\bF^{S,\bdeta},\mu)$ it remains to argue that $\tau$ is UI, which is equivalent to $\lim_{K\rightarrow +\infty} K\prob(\sup_{t \geq 0} |S_{\tau\wedge t}|\ge  K)=0$, see e.g. \citet{AzemaGundyYor:79}. We first show that $$\lim_{K\rightarrow +\infty} K\prob(\sup_{t \geq 0} S_{\tau\wedge t}\ge  K)=\lim_{K\rightarrow +\infty} K \prob(S^*_\tau\ge K)=0.$$ Because $\{S_{\tau\wedge t}\}$ never visits states outside any interval that contains the support of $\mu$, we only need to prove this when $\bar x=+\infty$, and hence $\lim_{y\rightarrow +\infty}b_\mu(y)=+\infty$, and taking $K\in \bN$.
	
	By \eqref{eq:ProbHittingn} and the construction of $\tau$, we see that for $n\in\bN$, $ n<\bar x$, we have
	\begin{align}
	&\prob(S^*_\tau\ge n)= \bP (\tau \geq \hit_n) = \bar \mu(x^n_{m_n+1}) \frac{\psi_\mu(x^n_{m_n+1}) - x^n_{m_n+1}}{n-x^n_{m_n+1}} \notag \\
	& = \bar \mu(b_\mu(n)) \frac{\psi_\mu(b_\mu(n)) - b_\mu(n)}{n-b_\mu(n)} = \bar \mu(b_\mu(n)+) \frac{\psi_\mu(b_\mu(n)+) - b_\mu(n)}{n-b_\mu(n)},\label{eq:AYdistrmax}
	\end{align}
where the last equality is the case because
\begin{align*}
  &\bar \mu(b_\mu(n)) \psi_\mu(b_\mu(n)) - \bar \mu(b_\mu(n)) b_\mu(n)\\
   &= \int_{[b_\mu(n),+\infty)}y\mu(dy)  -\mu(\{b_\mu(n)\})b_\mu(n)- \bar \mu(b_\mu(n)+) b_\mu(n)\\
   & = \int_{(b_\mu(n),+\infty)}y\mu(dy)- \bar \mu(b_\mu(n)+) b_\mu(n) = \bar \mu(b_\mu(n)+)\big(\psi_\mu(b_\mu(n)+) - b_\mu(n)\big);
\end{align*}
consequently,

	\begin{align*}
	n \prob(S^*_\tau\ge n) & = \bar \mu(b_\mu(n)+) \frac{\psi_\mu(b_\mu(n)+) - b_\mu(n)}{n-b_\mu(n)}n \\
	&= \bar \mu(b_\mu(n)+) \frac{\psi_\mu(b_\mu(n)+) - n}{n-b_\mu(n)}n + \bar \mu(b_\mu(n)+)n.
	\end{align*}
	For $a<c<d$ the function $y\to y(d-y)/(y-a)$ attains its maximum on $[c,d]$ in $y=c$. Taking
	$a=b_\mu(n)<c=\psi_\mu(b_\mu(n))\leq y=n<d=\psi_\mu(b_\mu(n)+)$, we can bound the first term by
	\begin{align*}
	& \bar \mu(b_\mu(n)+)  \frac{\psi_\mu(b_\mu(n)+) - n}{n-b_\mu(n)}n
	\leq \bar \mu(b_\mu(n)+) \frac{\psi_\mu(b_\mu(n)+) - \psi_\mu(b_\mu(n))}{\psi_\mu(b_\mu(n))-b_\mu(n)}\psi_\mu(b_\mu(n))\\
	& =\mu(\{b_\mu(n)\})\psi_\mu(b_\mu(n))
	\leq \bar \mu(b_\mu(n))\psi_\mu(b_\mu(n)) =\sum_{y\ge b_\mu(n)}y\mu(\{y\}),
	\end{align*}
	which goes to zero with $n\to \infty$ since $\mu\in \cM_0(\bZ)$. Similarly, $\psi_\mu(b_\mu(n)+)>n$ gives
	$$\bar\mu(b_\mu(n+))n\leq \bar\mu(b_\mu(n+))\psi_\mu(b_\mu(n)+)=\sum_{y>b_\mu(n)}y\mu(\{y\})\stackrel{n\to\infty}{\longrightarrow}0$$
	and we conclude that $n\bP(S^*_\tau\ge n)\to 0$ as $n\to \infty$.
	It remains to argue that 
	\begin{align*}
	\lim_{n\rightarrow +\infty} n\prob(\inf_{t \geq 0} S_{\tau \wedge t}\le -n)=0.
	\end{align*}
	This is trivial if $\underline{x}>-\infty$. Otherwise $b_\mu(0)=\underline{x}=-\infty$ and $x^0_i$'s are infinitely many. For $n\in\bN$, by the construction of $\tau$, $\inf_{t \geq 0} S_{t \wedge \tau}\le- n$ implies that $S$ visits $-n$ before hitting 1 and $S$ is not stopped at any $x^0_i>-n$. Denote the $i_n:=\sup\{i\ge 1|x^0_i> -n\}$ and note that $i_n\to\infty$ as $n\to \infty$. By construction, the probability that $S$ does not stop at $x^0_i$ given that $S$ reaches $x^0_i$ is $f^0_i/f^0_{i-1}$, $i\ge 1$. On the other hand, the probability that $S$ visits $-n$ before hitting 1 is $1/(n+1)$. Therefore, the probability that $S$ visits $-n$ before hitting 1 and $S$ is not stopped at any $x^0_i>-n$ is
	\begin{align*}
	\frac{1}{n+1}\prod_{i=1}^{i_n}f^0_i/f^0_{i-1} = \frac{1}{n+1}f^0_{i_n}.
	\end{align*}
	From \eqref{eq:DefineRho0}, $\lim_{k\rightarrow +\infty}f^0_k=0$ since $\mu$ has a finite first moment. Therefore,
	\begin{align*}
	\limsup_{n\rightarrow +\infty}n \prob(\inf_{t \geq 0} S_{\tau \wedge t}\le -n)\le \limsup_{n\rightarrow +\infty}n\frac{1}{n+1}f^0_{i_n}=0.
	\end{align*}
	The above concludes the proof of $\tau\in \SEP(\bF^{S,\bdeta},\mu)$.
	
	While \eqref{eq:AYoptimality} may be deduced from known bounds, as explained before, we provide a quick self--contained proof. Fix any $n\geq 1$ and $\sigma\in \SEP(\bF,\mu)$. When $\bar x<+\infty$, by UI, we have $ \prob(S^*_\tau\ge \bar x) =\prob(S_\tau= \bar x)=\mu(\{\bar x\}) =  P(S_\sigma=\bar x)= \prob(S^*_\sigma\ge \bar x) $ and $ \prob(S^*_\tau\ge n) = \prob(S^*_\sigma\ge n)=0$ for any $n>\bar x$.
	
	Next, by Doob's maximal equality and the UI of $\sigma$, $\mathbb{E}[(S_\sigma - n) {\bf 1}_{S^*_\sigma \geq n}] = 0$ and hence, for $k\leq n\in \bN$,
	\begin{align}
	0 & = \mathbb{E}[(S_\sigma - n) {\bf 1}_{S^*_\sigma \geq n} ]  = \mathbb{E}[(S_\sigma - n) {\bf 1}_{S_\sigma \geq k}] + \mathbb{E}[(S_\sigma - n) ({\bf 1}_{S^*_\sigma \geq n} - {\bf 1}_{S_\sigma \geq k} )] \notag\\
	& \leq \mathbb{E}[(S_\sigma - n) {\bf 1}_{S_\sigma \geq k}] + (k - n) \mathbb{E}[ {\bf 1}_{S^*_\sigma \geq n} {\bf 1}_{S_\sigma < k}] - (k - n) \mathbb{E}[{\bf 1}_{S^*_\sigma < n} {\bf 1}_{S_\sigma \geq k} ] \notag\\
	&  = \mathbb{E}[(S_\sigma - n) {\bf 1}_{S_\sigma \geq k}] + (k-n) \mathbb{E}[{\bf 1}_{S^*_\sigma \geq n} - {\bf 1}_{S_\sigma \geq k} ] \notag \\
	&= \mathbb{E}[(S_\sigma - k) {\bf 1}_{S_\sigma \geq k}]- (n-k) \prob(S^*_\sigma\geq n).\label{eq:DoobMaximal}
	\end{align}
	Considering $\bN\ni n<\overline{x}$ and $k=b_\mu(n)<n$, and recalling \eqref{eq:AYdistrmax},  we obtain
	\begin{align*}
	\prob(S^*_\sigma\geq n) &\le \frac{\mathbb{E}[(S_\sigma -b_\mu(n)) {\bf 1}_{S_\sigma \geq b_\mu(n)}]}{n-b_\mu(n)}
	= \frac{\bar \mu(b_\mu(n))\left(\psi_\mu(b_\mu(n)) - b_\mu(n)\right)}{n-b_\mu(n)}\\
	&= \prob(S^*_\tau\geq n).
	\end{align*}

\subsection{Proof of \eqref{eq:AYlikestoppingtime}}\label{sec:proof_keyrep}
	First, we show the inductive step: we prove that \eqref{eq:AYlikestoppingtime} holds for $j=n<\bar x$ given that it holds for $j=0,\dots, n-1$. Because \eqref{eq:AYlikestoppingtime} is true for $j=n-1$, we obtain
	\begin{align}
	\mathbb{P}(\tau \geq \hit_n) &= 1 - \mathbb{P}(\tau < \hit_n) = 1- \left(\sum_{y < x^{n-1}_1} \mu(\{y\})\right) - \bar {\mu}(x^{n-1}_1) \frac{n-\psi_{\mu}(x^{n-1}_1)}{n-x^{n-1}_1} \notag \\
	&=\bar {\mu}(x^{n-1}_1)- \bar {\mu}(x^{n-1}_1) \frac{n-\psi_{\mu}(x^{n-1}_1)}{n-x^{n-1}_1}  =  {\bar {\mu}(x^{n-1}_1)}\frac{\psi_{\mu}(x^{n-1}_1)-x^{n-1}_1}{n-x^{n-1}_1} \notag\\
	&=  {\bar {\mu}(x^{n}_{m_n+1})}\frac{\psi_{\mu}(x^{n}_{m_n+1})-x^{n}_{m_n+1}}{n-x^{n}_{m_n+1}} = \Gamma_n,\label{eq:ProbHittingn}
	\end{align}
	where $\Gamma_n$ is defined as in \eqref{eq:DefineRho}. Recalling that $x^{n}_{m_n+1}=b_\mu(n)\le n-1<\bar x$ and $\psi_\mu(x)<x$ for $x<\bar x$, we conclude that $\mathbb{P}(\tau \geq \hit_n)>0$.
	
	Consider first the case when $m_{n}=0$, so that $x^{n}_1 = x^{n}_{m_n+1}=x^{n-1}_1$
	 and $\tau$ stops if $x^n_1$ is hit between $\hit_n$ and $\hit_{n+1}$. Consequently,
	\begin{align*}
	&\mathbb{P}(S_{\tau} = x^{n}_1,\tau < \hit_{n+1})
	=\mathbb{P}(S_{\tau} = x^{n}_1 ,\tau < \hit_n)+ \mathbb{P}(S_{\tau} = x^{n}_1, \tau < \hit_{n+1} | \tau \geq \hit_n) \cdot \mathbb{P}(\tau \geq \hit_n)\\
	&=\mathbb{P}(S_{\tau} = x^{n-1}_1, \tau < \hit_n)+ \frac{1}{n+1-x^{n}_1} \cdot {\bar {\mu}(x^{n}_{m_n+1})}\frac{\psi_{\mu}(x^{n}_{m_n+1})-x^{n}_{m_n+1}}{n-x^{n}_{m_n+1}}\\
	&=\bar {\mu}(x^{n-1}_1) \frac{n-\psi_{\mu}(x^{n-1}_1)}{n-x^{n-1}_1}+ \frac{1}{n+1-x^{n}_1} \cdot {\bar {\mu}(x^{n}_{m_n+1})}\frac{\psi_{\mu}(x^{n}_{m_n+1})-x^{n}_{m_n+1}}{n-x^{n}_{m_n+1}}\\
	&= \bar {\mu}(x^n_1) \frac{n-\psi_{\mu}(x^n_1)}{n-x^n_1}+\frac{1}{n+1-x^{n}_1} \cdot \bar {\mu}(x^n_1) \frac{\psi_{\mu}(x^n_1)-x^n_1}{n-x^n_1}\\
	&= \bar {\mu}(x^{n}_1) \frac{n+1-\psi_{\mu}(x^{n}_1)}{n+1-x^{n}_1},\quad \textrm{ and we conclude that \eqref{eq:AYlikestoppingtime} holds for $j=n$.}
	\end{align*}
	
	Next, we consider the case in which $m_n\ge 1$.
	A direct calculation yields
	\begin{align*}
	&\mu(\{x^{n}_{m_n+1}\}) - \bar \mu(x^{n}_{m_n+1})\frac{n-\psi_\mu(x^{n}_{m_n+1})}{n-x^{n}_{m_n+1}}\\
	&= \frac{1}{n-x^{n}_{m_n+1}}\left[\sum_{y>x^n_{m_n+1}} y\mu(\{y\}) - n\mu((x^n_{m_n+1},+\infty))\right] \\
	&= \frac{\mu((x^n_{m_n+1},+\infty))}{n-x^{n}_{m_n+1}}\left[\psi_\mu(x^n_{m_n+1}+)-n\right]=\frac{\mu((x^n_{m_n+1},+\infty))}{n-x^{n}_{m_n+1}}\left[\psi_\mu(b_\mu(n)+)-n\right]> 0,
	\end{align*}
	where the inequality follows from $x^n_{m_n+1}=b_\mu(n) <n<\bar x$ and $\psi_\mu(b_\mu(y)+)>y$ for any $y<\bar x$. Consequently, $g_{m_n+1}>0$.
	On the other hand, because $\mu(\{y\})=0$ for any $y$ that is not in the range of $b_\mu$, we conclude that
	\begin{align*}
	&\sum_{k=2}^{m_n+1}(n+1-x^{n}_{k})\mu(\{x^n_k\})=\sum_{x^n_{m_n+1}\le y<x^{n}_1}(n+1-y)\mu(\{y\})\\
	&=(n+1)(\bar \mu(x^n_{m_n+1})-\bar \mu(x^n_{1})) - \left[\psi_\mu(x^n_{m_n+1})\bar \mu (x^n_{m_n+1}) - \psi_\mu(x^n_{1})\bar \mu (x^n_{1})\right]\\
	&= (n+1-\psi_\mu(x^n_{m_n+1}))\bar \mu(x^n_{m_n+1}) - (n+1-\psi_\mu(x^n_{1}))\bar \mu(x^n_{1}).
	\end{align*}
	Consequently, we have
	\begin{align*}
	f_1^n=&\sum_{k=2}^{m_n+1}g_k\\
	 = &\frac{1}{\Gamma_n}\left[\sum_{k=2}^{m_n+1}(n+1-x^{n}_{k})\mu(\{x^n_k\}) - (n+1-x^{n}_{m_n+1}) \bar \mu(x^{n}_{m_n+1})\frac{n-\psi_\mu(x^{n}_{m_n+1})}{n-x^{n}_{m_n+1}} \right]\\
	=&\frac{1}{\Gamma_n}\left[\bar\mu(x^n_{m_n+1})\frac{\psi_\mu(x^n_{m_n+1})-x^n_{m_n+1}}{n-x^n_{m_n+1}}-(n+1-\psi_\mu(x^n_{1}))\bar \mu(x^n_{1})\right]\\
	=& 1-\frac{(n+1-\psi_\mu(x^n_{1}))\bar \mu(x^n_{1})}{\prob(\tau\ge \hit_n)}\le 1,
	\end{align*}
	where the last inequality holds because $\psi_\mu(x^n_{1}) =  \psi_\mu\big(\min(b_\mu(n+1),n)\big)\le \psi_\mu\big(b_\mu(n+1)\big)\le n+1$.
	It follows, since $g^n_k>0,k=2,\dots,m_n+1$, that $f_k$ is strictly decreasing in $k=1,2,\dots, m_n$ with $f^n_{m_n}>0$ and $f^n_1\le 1$. Consequently, $\rho^n_k$ is well defined and $\rho^n_k\in[0,1)$, $k=1,\dots, m_n$. Recall $\rho^n_{m_n+1} = 1 = 1-(f^n_{m_n+1}/f^n_{m_n})$. Set $x^n_0:=n$. Then, for each $k=1,\dots, m_n+1$,
	\begin{align*}
	&\mathbb{P}(S_{\tau} = x^n_k , \hit_n \leq \tau < \hit_{n+1})
	= \mathbb{P}(S_{\tau} = x^n_k , \tau < \hit_{n+1} | \tau \geq \hit_n) \cdot \mathbb{P}(\tau \geq \hit_n)\\
	& =\left[\prod_{j=1}^{k-1} \left(\frac{n+1-x^{n}_{j-1}}{n+1-x^n_{j}}(1-\rho^n_{j})\right)\right]\frac{n+1-x^{n}_{k-1}}{n+1-x^n_{k}}\rho^n_{k}\mathbb{P}(\tau \geq \hit_n)\\
	& = \frac{1}{n+1-x^n_{k}}\left[\prod_{j=1}^{k-1} (1-\rho^n_{j})\right]\rho^n_{k}\mathbb{P}(\tau \geq \hit_n)
	= \frac{1}{n+1-x^n_{k}}(f^n_{k-1}-f^n_k)\mathbb{P}(\tau \geq \hit_n).
	\end{align*}
	Therefore, recalling the definition of $f^n_k$ and $g^n_k$,
	for $k=2,\dots, m_n$, we have $\mathbb{P}(S_{\tau} = x^n_k,\hit_n \leq \tau < \hit_{n+1}) = \mu(\{x^n_k\})$.
	Further, $\mathbb{P}(S_{\tau} = x^n_{m_n+1},\hit_n \leq \tau < \hit_{n+1}) = \mu(\{x^n_{m_n+1}\}) - \bar \mu(x^{n}_{m_n+1})\frac{n-\psi_\mu(x^{n}_{m_n+1})}{n-x^{n}_{m_n+1}}$ and $\mathbb{P}(S_{\tau} = x^n_1,\hit_n \leq \tau < \hit_{n+1}) = \frac{n+1-\psi_{\mu}(x^{n}_{1})}{n+1-x^n_{1}}\bar \mu(x^n_1)$
	and we verify that \eqref{eq:AYlikestoppingtime} holds for $j=n$.

	We move on to showing the inductive base step: we prove \eqref{eq:AYlikestoppingtime} holds for $j=0$. When $\underline{x}>-\infty$, $x^0_i$'s are finitely many and the proof is exactly as above. When $\underline{x}=-\infty$, by definition, $g^0_k>0$ because $x^0_k\le x^0_1 \le 0$ and $\mu(\{x^0_k\})>0$. Recalling that $\mu(\{y\})=0$ for any $y$ that is not in the range of $b_\mu$, we have
	\begin{align*}
	\sum_{k=2}^\infty g^0_k &= \sum_{k=2}^\infty (1-x^{0}_{k})\mu(\{x^0_k\}) = \sum_{y< x^0_1}(1-y)\mu(\{y\}) = 1-\bar \mu(x^0_1) - \sum_{y< x^0_1}y\mu(\{y\})\\
	& =1-\bar \mu(x^0_1) + \sum_{y\ge  x^0_1}y\mu(\{y\}) =  1- \bar \mu(x^0_1) \left(1-\psi_\mu(x^0_1)\right).
	\end{align*}
	Because $x^0_1= \min(b_\mu(1),0)$ and $\psi_\mu(b_\mu(y))\le y$ for any $y$, we conclude that $\psi_\mu(x^0_1) \le \psi_\mu(b_\mu(1))\le 1$. In addition, $\bar \mu(x^0_1)<1$, showing that $\sum_{k=2}^\infty g^0_k<1$. Therefore, $f^0_k$'s are well defined, positive, and strictly decreasing in $k$, and $f^0_1< 1$. Therefore, $\rho_k^0\in (0,1)$ for each $k\ge 1$. Following the same arguments as previously, one concludes that \eqref{eq:AYlikestoppingtime} holds for $j=0$.

	\section{Examples}\label{se:Example}
We end this paper with an explicit computation of our two embeddings for two examples.

\subsection{Optimal Gambling Strategy}
The first example is a measure $\mu$ arising naturally from the casino gambling model studied in \citet{HeEtal2014:StoppingStrategies}. Therein, a gambler whose preferences are represented by cumulative prospect theory \citep{TverskyKahneman1992:CPT} is repeatedly offered a fair bet in a casino and decides when to stop gambling and exit the casino. The optimal distribution of the gambler's gain and loss at the exit time is a certain $\mu\in \cM_0(\bZ)$, which may be characterised explicitly, see \citet[Theorem 2]{HeEtal2014:StoppingStrategies}. 	
	 With a set of reasonable model parameters\footnote{Specifically: $\alpha_+ = 0.6$, $\delta_+ = 0.7$, $\alpha_- = 0.8$, $\delta_- = 0.7$, and $\lambda = 1.05$.}, we obtain
	\begin{align}\label{eq:DistributionExample}
	\mu(\{n\})=
	\begin{cases}
	0.4465\times ((n^{0.6}-(n-1)^{0.6})^{\frac{10}{3}}-((n+1)^{0.6}-n^{0.6})^{\frac{10}{3}}),& n\ge 2,\\
	0.3297,  & n=1, \\
	0.6216, & n=-1,\\
	0, & \text{otherwise}.
	\end{cases}
	\end{align}
	We first exhibit the randomized Markovian stopping time $\tau(\vecr)$ of Theorem \ref{thm:reformulate}. Using the algorithm given in Section \ref{subse:ConstructionRandPI} we compute $r^i$, the probabilities of a coin tossed at $S_t=i$ turning up tails, for all $i\in \mathbb{Z}$:
	\begin{align*}
	r^1 &= 0.4040,\;  r^2 = 0.0600, \; r^3 = 0.0253, \; r^4 = 0.0140,\; r^5 = 0.0089,\\
	r^6 &= 0.0061,\; r^7 = 0.0045,\; r^8 = 0.0034,\; r^9 = 0.0027,\; r^{10} = 0.0021,\dots
	\end{align*}
	Note that $\mu(\{0\})=0$ and $\mu(\{n\})=0,n\le -2$; so one does not stop at 0 and must stop upon reaching $-1$. The stopping time $\tau(\vecr)$ is illustrated in the left pane of Figure \ref{fi:PIAY}: $S$ is represented by a recombining binomial tree. Black nodes stand for ``stop", white nodes stand for ``continue", and grey nodes stand for the cases in which a random coin is tossed and one stops if and only if the coin turns up tails. The probability that the random coin turns tails is shown on the top of each grey node.
	
	Next we follow Theorem \ref{th:SkorokhodembeddingAYstoppingtime} to construct a randomized Az\'ema-Yor stopping time $\tau_\mu^{AY}$ embedding $\mu$. To this end, we compute $x^n_k$'s and $\rho^n_k$'s, which stand for the drawdown levels that are set after reaching maximum $n$ and the probabilities that the coins tossed at these levels turn up tails, respectively:
	\begin{align*}
	m_{0} &= 0,\; x^0_1 = -1,\quad m_1=1,\; x^1_1=1,\; \rho^1_1=0.2704,\; x^1_2=-1,\\
	m_2 &= 0,\; x^2_1=1,\quad m_3=0,\; x^3_1=1,\quad m_4=0,\; x^4_1=1,\\
	m_5&=1,\; x^5_1=2,\;\rho^5_1=0.0049,\; x^5_2=1,\quad m_6=0,\; x^6_1=1,\dots
	\end{align*}
	The stopping time is then illustrated in the right pane of Figure \ref{fi:PIAY}: $S$ is represented by a non-recombining binomial tree. Again, black nodes stand for ``stop", white nodes stand for ``continue", and grey nodes stand for the cases in which a random coin is tossed and one stops if and only if the coin turns up tails. The probability that the random coin turns up tails is shown on the top of each grey node.
	
	By definition, $\tau(\vecr)$ is Markovian: at each time time $t$, the decision to stop depends only on the current state and an independent coin toss. However, to implement the strategy, one needs to toss a coin most of the times. In contrast $\tau_\mu^{AY}$ requires less independent coin tossing: e.g.\ in the first five periods at most one such a coin toss, but it is path-dependent. For instance, consider $t=3$ and $S_t=1$. If one reaches this node along the path from (0,0), through (1,1) and (2,2), and to (3,1), then she stops. If one reaches this node along the path from (0,0), through (1,1) and (2,0), and to (3,1), then she continues. Therefore, compared to the randomized Markovian strategy, the randomized Az\'ema-Yor strategy may involve less independent coin tosses but is typically path-dependent\footnote{Indeed, \cite{HeHuOblojZhou2016:Randomization}
showed that, theoretically, any path-dependent strategy is equivalent to a randomization of  Markovian strategies.} as it considers relative loss when deciding if to stop or not.
	
	\begin{figure}
		\centering\begin{minipage}[t]{0.40\textwidth}
			\hspace*{-0.5cm}\scalebox{0.53}[0.63]{\includegraphics{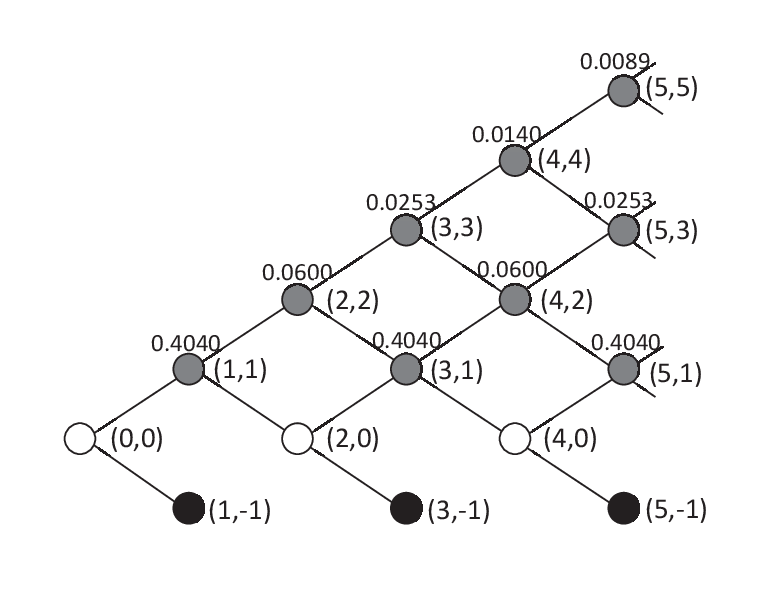}}
		\end{minipage}
		\hfill
		\begin{minipage}[t]{0.5\textwidth}
			\hspace*{-0.3cm}\scalebox{0.5}[0.4]{\includegraphics{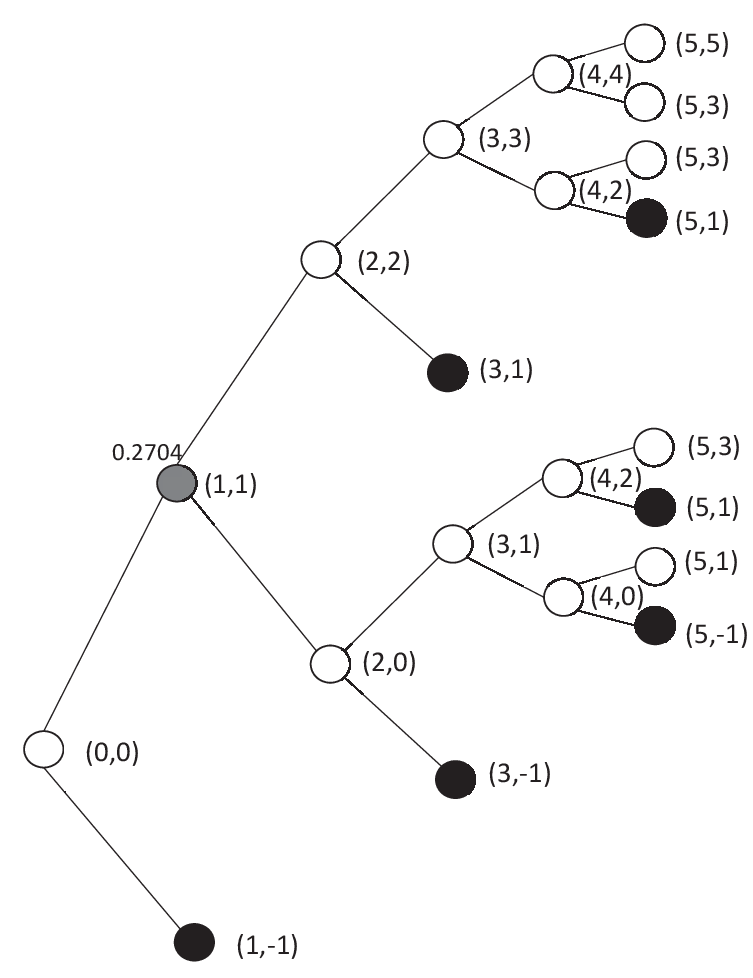}}
		\end{minipage}
		\caption{Randomized Markovian stopping time $\tau(\vecr)$ (left pane) and randomized Az\'ema-Yor stopping time $\tau_\mu^{AY}$ (right pane) embedding probability measure $\mu$ in \eqref{eq:DistributionExample} into the simple symmetric random walk $S$. Black nodes stand for ``stop", white nodes stand for ``continue", and grey nodes stand for the cases in which a random coin is tossed and one stops if and only if the coin turns up tails. Each node is marked on the right by a pair $(t,x)$ representing time $t$ and $S_t=x$. Each grey node is marked on the top by a number showing the probability that the random coin tossed at that node turns up tails.}\label{fi:PIAY}
	\end{figure}
	
\subsection{Mixed Geometric Measure}
The second example is a mixed geometric measure $\mu$ on $\bZ$ with
\begin{align}\label{eq:DistributionExamplegeometric}
\mu(\{n\})=
\begin{cases}
\gamma_+\left[ q_+(1-q_+)^{n-1}\right] , & n \geq 1,\\
1 - \gamma_+ - \gamma_-, & n=0,\\
\gamma_-\left[q_-(1-q_-)^{-n-1}\right]  , & n \leq -1,
\end{cases}
\end{align}
where $\gamma_\pm \ge 0$, $q_\pm\in (0,1)$, $\gamma_++\gamma_-\le 1$, and $\gamma_+/ q_+ = \gamma_- / q_-$ so that $\mu \in \cM_0(\bZ)$.

The randomized Markovian stopping time that embeds $\mu$ given by \eqref{eq:DistributionExamplegeometric} can be derived analytically. Indeed, according to the algorithm given in Section \ref{subse:ConstructionRandPI}, the probability of a coin tossed at $S_t=i$ turning up tails is
\begin{align*}
  r^i =
  \begin{cases}
  q_+^2/[(1-q_+)^2 + 1], & i \geq 1,\\
  (1 - \gamma_+ - \gamma_-)/[1 - \gamma_+ - \gamma_- + 2(\gamma_+/q_+)], &i = 0, \\
  q_-^2/[(1 - q_-)^2 + 1], & i \leq -1.
  \end{cases}
\end{align*}
The randomized Az\'ema-Yor stopping time that embeds $\mu$ given by \eqref{eq:DistributionExamplegeometric} can also be derived analytically. Because the formulae for $x^n_k$'s and $\rho^n_k$'s are tedious, we chose not to present them here. Instead, we illustrate the two embeddings in Figure \ref{fi:PIAYMG} by setting $q_+ = \gamma_+ = 5/12$ and $q_- = \gamma_- = 13/24$. As in the previous example, the randomized Az\'ema-Yor stopping time involves less randomization than the randomized Markovian stopping time at the cost of being path-dependent.

\begin{figure}
\centering\begin{minipage}[t]{0.45\textwidth}
\scalebox{0.45}[0.6]{\includegraphics{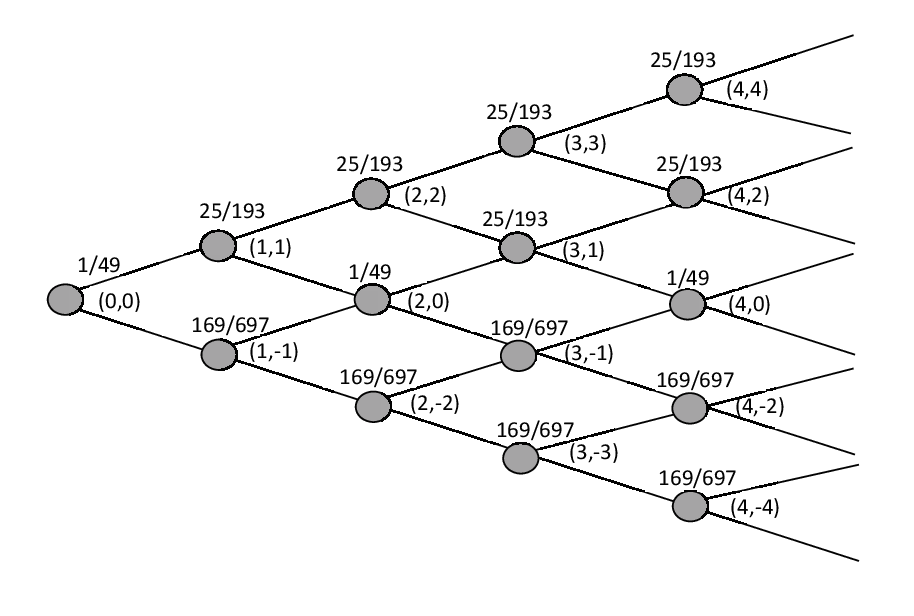}}
\end{minipage}
\hfill
\begin{minipage}[t]{0.45\textwidth}
\scalebox{0.65}[0.8]{\includegraphics{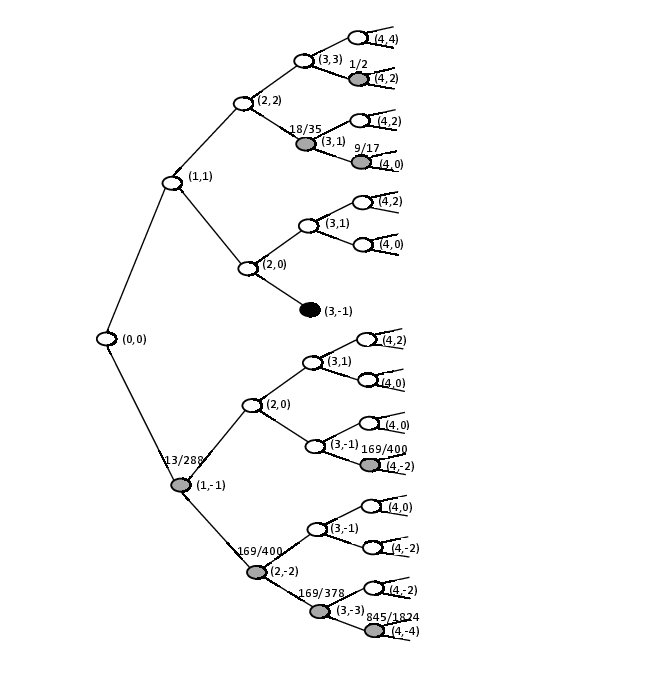}}
\end{minipage}
\caption{Randomized Markovian stopping time (left-panel) and randomized Az\'ema-Yor stopping (right panel) embedding probability measure \eqref{eq:DistributionExamplegeometric} with $q_+ = \gamma_+ = 5/12$ and $q_- = \gamma_- = 13/24$ into the random walk $\{S_t\}$. Black nodes stand for ``stop", white nodes stand for ``continue", and grey nodes stand for the cases in which a random coin is tossed and one stops if and only if the coin turns tails. Each node is marked on the right by a pair $(t,x)$ representing time $t$ and $S_t=x$. Each grey node is marked on the top by a number showing the probability that the random coin tossed at that node turns up tails.}\label{fi:PIAYMG}
\end{figure}


\end{document}